\documentclass[12pt,a4paper]{article}
\usepackage{mathptmx, enumerate, amsmath, amsthm, amscd, amsfonts, amssymb, graphicx, color}
\usepackage[bookmarksnumbered, colorlinks, plainpages]{hyperref}

\newtheorem{thm}{Theorem}[section]
\newtheorem{cor}[thm]{Corollary}
\newtheorem{lem}[thm]{Lemma}
\newtheorem{prop}[thm]{Proposition}
\theoremstyle{definition}

\theoremstyle{remark}
\newtheorem{rem}[thm]{Remark}
\newtheorem{ex}[thm]{Example}
\numberwithin{equation}{section}

\title{ Around Strongly Operator Convex Functions\\}

\author{Nahid Gharakhanlu\footnote{nahid.gh80@gmail.com}~ and ~Mohammad Sal Moslehian
	\footnote{moslehian@um.ac.ir}\\
 \\
\small Department of Pure Mathematics, Ferdowsi University of Mashhad,\\ \small P. O. Box 1159, Mashhad, 91775, Iran.
 }

\begin{document}
 \maketitle

\noindent \abstract In this paper, we obtain the subadditivity inequality of strongly operator convex functions on $(0, \infty)$ and $(-\infty,0)$. Applying the properties of operator convex functions, we deduce the subadditivity property of operator monotone  functions on $(0, \infty)$. We give new operator inequalities involving  strongly operator convex functions on an interval and the weighted operator means. We also investigate relations between strongly operator convex functions and Kwong functions on $(0, \infty)$.  Moreover, we study the strongly operator convex functions on $(a, \infty)$  with $a>-\infty$ and also on left half line $(-\infty,b)$ with $b< \infty$. We show that any non-constant strongly operator convex function on $(a, \infty)$ is strictly operator decreasing, and any non-constant strongly operator convex function on $(-\infty,b)$ is strictly operator monotone. Consequently,  if $g$ is a strongly operator convex function on $(a, \infty)$ or $(-\infty,b)$, we estimate lower bounds of $|g(A)-g(B)|$ whenever $A-B>0$.

\noindent {\textbf{keywords:}}  Strongly operator convex; Operator monotone function; Integral representation; Subadditivity; Operator inequality.

\section{Introduction}\label{sec1}

Let $\mathbb{B}(\mathcal{H})$ denote the algebra of all bounded linear operators on a complex Hilbert space $(\mathcal{H}, \langle \cdot,\cdot\rangle)$ equipped with the operator norm $\Vert \cdot \Vert$. 
 An operator $A\in \mathbb{B}(\mathcal{H})$ is called positive if $\langle Ax , x\rangle \geq 0$ holds for every $x \in \mathcal{H}$ and then we write $A\geq 0$. We write $A>0$ when the operator $A$ is positive and invertible (called also strictly positive). When $-A$ positive, we say that $A$ is negative. For self-adjoint operators $A, B\in \mathbb{B}(\mathcal{H})$, we say $A\leq B$ if $B-A\geq 0$.  
 Let $f$ be a  continuous real-valued function defined on an interval $J$. For a bounded self-adjoint operator $A$ on a Hilbert space whose spectrum ${\rm sp}(A)$ is in $J$, $f(A)$ stands for the continuous functional calculus. A function $f$ is called operator monotone if for all self-adjoint operators $A, B\in \mathbb{B}(\mathcal{H})$ with spectra in $J$, $A\leq B$ implies $f(A)\leq f(B)$. A continuous real-valued function $f$ is said to be operator convex if 
\begin{equation*}
f((1-\alpha) A+ \alpha B)\leq (1-\alpha) f(A)+ \alpha f(B),
\end{equation*}
 for all $\alpha \in [0,1]$. A function $f$ is called operator decreasing (operator concave) if so is $-f$. For more details, we refer the readers to \cite[Chapter V]{5}. Throughout the paper, we assume that all functions are continuous.
 
Regarding the subadditivity and  supperadditivity properties, Moslehian and Najafi obtained the subadditivity condition of operator monotone functions as follows. 
\begin{thm}\emph{\cite[Theorem 2.1]{MN}}\label{3}
Let $f$ be a non-negative operator monotone function on $[0, \infty)$ and $A, B$ be positive operators. Then 
\begin{equation*}
f(A+B)\leq f(A)+f(B),
\end{equation*}
 if and only if $AB+BA\geq 0$.
\end{thm}

We recall that $AB+BA$ is called the symmetrized product of $A, B \in \mathbb{B}(\mathcal{H})$. If $A$ and $B$ are positive operators, then so is $A+B$. In general, the symmetrized product of two positive operators is not positive. For example, consider two positive operators $A=\begin{bmatrix} 1 & 0 \\ 0 & 0 \end{bmatrix}$ and $B= \begin{bmatrix} 1 & 1 \\ 1 & 1 \end{bmatrix}$. It can be verified that $AB+BA$ is not a positive operator.

Let $f$ be a non-constant operator monotone functions on $[0, \infty)$. Then, it is strictly operator monotone in the sense that If $A>B>0$, then $f(A)>f(B)$; see \cite[Theorem 2.1]{Furuta2} and \cite[Proposition 2.2]{MN2012}. Furuta estimated lower bounds of $f(A)-f(B)$ for all non-constant operator monotone functions on $[0, \infty)$ whenever $A-B>0$ as follows.
\begin{thm}\emph{\cite[Theorem 3.1]{Furuta}}\label{17}
Let $A$ and $B$ be positive invertible operators such that $A-B\geq m>0$ for some scalar $m>0$. Then
\begin{equation*}
f(A)-f(B)\geq f(\Vert B\Vert +m)-f(\Vert B \Vert)\geq f(\Vert A \Vert)-f(\Vert A\Vert -m)> 0,
\end{equation*}
for all non-constant operator monotone functions on $[0, \infty)$. 
\end{thm}
 
 In  \cite{Davis}, Davis proved that a function $g$ is an operator convex  on $J$ if and only if $P g(PAP) P\leq P g(A) P$ for every $A$ with spectrum in $J$ and for every orthogonal projection $P$. If $Pg(PAP)P\leq g(A)$, then $g$ is called a strongly operator convex function on $J$ and then  we write $g\in\textbf{SOC}(J)$ \cite{Brown2, Brown1}.  We observe  that a strongly operator convex function is operator convex and non-negative. Brown and Uchiyama  obtained the following characterization of strongly operator convex functions which plays an essential role in our investigation.

\begin{thm}\emph{\cite[Proposition 2.3]{BU}}\label{1}
Let $g\neq 0$ be a non-constant continuous function, $a> -\infty$ and $b< \infty$. Then 
\begin{enumerate}[(i)]
\item $g\in$ {\rm{\textbf{SOC}}}$(a, \infty)$ if and only if $g(t)>0$ and $g$ is operator decreasing on $(a, \infty)$. In this case, $g$ is represented as
\begin{equation}\label{eq-soc}
g(t)=g(\infty) + \int\limits_{(-\infty ,a]} \frac{1}{t-\lambda} d \mu (\lambda),
\end{equation}
where  $\lambda\in (-\infty ,a]$, $\int\limits_{(-\infty ,a]} \frac{d \mu (\lambda)}{1+\vert \lambda \vert} < \infty$, and $\mu$ is a positive measure on $(-\infty ,a]$.
\item $g\in$ {\rm{\textbf{SOC}}}$(-\infty, b)$ if and only if $g(t)>0$ and $g$ is operator monotone on $(-\infty, b)$. In this case, $g$ is represented as
\begin{equation}\label{eq-soc1}
g(t)=g(-\infty) + \int\limits_{[b,\infty)} \frac{1}{\lambda-t} d \mu (\lambda),
\end{equation}
in which  $\lambda\in [b,\infty)$, $\int\limits_{[b,\infty)} \frac{d \mu (\lambda)}{1+\vert \lambda \vert} < \infty$, and $\mu$ is a positive measure on $[b,\infty)$.
\end{enumerate}
\end{thm}

If $ 0\neq g\in$ {\rm{\textbf{SOC}}}$(0, \infty)$,  then $f(t)=-g(t)$ is a negative operator monotone function on $(0, \infty)$ and consequently,  Theorem \ref{17} with $f(t)=-g(t)$ on $(0, \infty)$ ensures the following result.
\begin{cor}\label{18}
Let $A$ and $B$ be positive invertible operators such that $A-B\geq m>0$ for some scalar $m>0$. Then
\begin{equation*}
g(B)-g(A)\geq g(\Vert B \Vert)-  g(\Vert B\Vert +m) \geq g(\Vert A\Vert -m)-g(\Vert A \Vert)> 0,
\end{equation*}
for all non-constant $ 0\neq g\in$ {\rm{\textbf{SOC}}}$(0, \infty)$.
\end{cor}

Let $A, B >0$ and $\alpha\in (0,1)$. We recall that the $\alpha$-weighted operator arithmetic and weighted operator harmonic mean of $A$ and $B$ are defined, respectively, by
\begin{align*}
A\nabla_{\alpha} B &= (1-\alpha)A+\alpha B,\\
A!_{\alpha} B &= \left( (1-\alpha)A^{-1}+\alpha B^{-1}\right)^{-1}.
\end{align*}
When $\alpha=\frac{1}{2}$,  $A\nabla_{\frac{1}{2}} B$ and $A!_{\frac{1}{2}} B$ are called the operator arithmetic mean and operator harmonic mean, and are denoted by  $A\nabla B$ and $A!B$, respectively. Brown and Uchiyama proved the following   inequalities between operator means and  strongly operator convex functions on an interval $J$.

\begin{thm}\emph{\cite[Proposition 2.1]{BU}}\label{8}
Let $g\neq 0$ be a continuous function on $J$ and $\alpha\in (0,1)$. Then $g\in$ {\rm{\textbf{SOC}}}$(J)$ if and only if $g(t)>0$ and 
\begin{equation*}
g(A)\nabla_{\alpha} g(B)-g(A\nabla_{\alpha} B)\geq \alpha (1-\alpha) (g(A)-g(B))(\alpha g(A)+(1-\alpha)g(B))^{-1} (g(A)-g(B)),
\end{equation*}
for all self-adjoint operators $A,B$ with spectra in $J$.
\end{thm}

\begin{thm}\emph{\cite[Theorem 3.1]{Uchiyama3}}\label{10}
Let $g\neq 0$ be  continuous function on $J$. Then the following are equivalent to each other.
\begin{enumerate}[(i)]
\item $g\in$ {\rm{\textbf{SOC}}}$(J)$
\item  $g(t)>0$ and $g(A\nabla B)\leq g(A) ! g(B) $
\item  $g(t)>0$ and 
\begin{equation*}
\begin{bmatrix} g(A) & g(A\nabla B) \\ g(A\nabla B) & g(B) \end{bmatrix} \geq \frac{1}{2}\begin{bmatrix} g(A\nabla B) & g(A\nabla B) \\ g(A\nabla B) & g(A\nabla B) \end{bmatrix},
\end{equation*}
\end{enumerate}
for all  self-adjoint operators $A,B$ with spectra in $J$.
\end{thm}

Let $f$ be a  real-valued function on an interval $(0, \infty)$. Kwong matrices, called also anti-L\"owner matrices, associated with $f$ are defined by
\begin{equation*}
K_{f}(t_{1}, t_{2},...,t_{n})=\left[ \frac{f(t_{i})+f(t_{j})}{t_{i}+t_{j}}     \right]_{i,j=1}^{n},
\end{equation*}
for any distinct real numbers $t_{1}, ...,t_{n}$ in $(0, \infty)$. A function $f$ is  called Kwong  if the Kwong  matrix $K_{f}(t_{1}, t_{2},...,t_{n})\geq 0$  for distinct real numbers $t_{1}, ...,t_{n}$ \cite[Chapter 5]{B1}.  If $f$ is a Kwong function on $(0, \infty)$, we write  $f\in \textbf{Kwong}(0, \infty)$. Kwong \cite{14} showd that if a function $f$ from $(0, \infty)$ into itself is operator monotone, then $f$ is  Kwong , and so is every non-negative operator decreasing function. The statement for non-negative operator  decreasing functions follows from the fact that  $f$ is Kwong if and only if $\frac{1}{f}$ is  Kwong.  Audenaert  showed the following relation between Kwong functions and operator decreasing functions.  

\begin{thm}\emph{\cite[Theorem 2.1]{17}}\label{2}
Let $f$ be a real-valued function on $(0,\infty)$.  If $f$ is a Kwong function  on $(0,\infty)$, then $\frac{f(\sqrt{t})}{\sqrt{t}}$ is a non-negative operator decreasing function on $(0,\infty)$.
\end{thm}

This paper is arranged as follows.  In section \ref{sec2},  Theorem \ref{5new} compares $g(A+B)$ with  $g(A)+g(B)$ for the strongly operator convex functions on $(-\infty,0)$. This leads to the subadditivity inequality of positive operator monotone functions on $(-\infty,0)$.  In comparison with Theorem \ref{3}, Theorem \ref{16} gives the subadditivity property of positive operator monotone functions on $(0, \infty)$ with sharp constant. Theorem \ref{9}, Lemma \ref{12} and Corollary \ref{13} establish a  generalization of Theorem \ref{10} involving the strongly operator convex functions on an interval and the weighted operator means.  In Proposition \ref{14}, we investigate relations between strongly operator convex functions and Kwong functions on $(0, \infty)$. Then, we deduce some positivity results for the symmetrized product of positive invertible operators. In section  \ref{sec3}, we deal with  strongly operator convex functions on right half line $(a, \infty)$ with  $a>-\infty$ and also left half line $(-\infty,b)$ with $b< \infty$. In Theorem \ref{20}, we show that any non-constant strongly operator convex function $g$ on $(a, \infty)$ is strictly operator decreasing and estimate lower bounds of $g(B)-g(A)$ when $A-B>0$. Theorem \ref{20} is an eventual extension of Corollary \ref{18}, and also deduce an extension of  Theorem \ref{17} in Corollary \ref{23}. Moreover, in Theorem \ref{21} and its corollaries, we study the behavior of strongly operator convex functions on $(-\infty,b)$. Consequently in Corollary \ref{24}, we estimate lower bounds of $f(A)-f(B)$  for any non-constant positive operator monotone functions $f$ on $(-\infty,b)$ whenever $A-B>0$.

\section{Strongly operator convex functions}\label{sec2}

By the operator decreasing property of $ g\in$ {\rm{\textbf{SOC}}}$(0, \infty)$, we have
\begin{prop}\label{25}
If $0\neq g\in$ {\rm{\textbf{SOC}}}$(0, \infty)$ and $A, B$ are positive invertible operators, then
\begin{equation*}
g(A+B)\leq g(A)+g(B).
\end{equation*}
\end{prop}

\begin{thm}\emph{\cite[Theorem 2.1]{Uchiyama1}}\label{Uchi}
Let $g$ be an operator convex function on $[0,\infty)$ with $g(0)\leq 0$. Then $AB+BA\geq 0$  if and only if 
\begin{equation*}
g(A+B)\geq  g(A)+g(B),
\end{equation*}
for all positive operators $A,B$ 
\end{thm}

Notice that the subadditivity property of operator convex functions does not hold in general. Proposition \ref{25} and \cite[Theorem 2.4]{Uchiyama} give the subadditivity of operator convex functions as follows.
  
\begin{cor}
Let $g$ be a positive operator convex function on $(0,\infty)$ with $\lim\limits_{t\to\infty}g(t)< \infty$ and $A, B$ be positive invertible operators. Then $g(A+B)\leq g(A)+g(B)$.
\end{cor}

If $0\neq g\in$ {\rm{\textbf{SOC}}}$(-\infty,0)$, then $g$ is a positive operator convex function on $(-\infty,0)$ \cite[Corollary 2.7]{Uchiyama}. Thus, we have
\begin{equation*}
g( \frac{A+B}{2})\leq g(A)+g(B),
\end{equation*}
for all negative invertible operators $A$ and $B$. Now, we may ask is it possible to compare $g(A+B)$ with $ g(A)+g(B)$ for a strongly operator convex function $g$ on $(-\infty,0)$? We give an affirmative answer as follows.

\begin{thm}\label{5new}
Let $0\neq g\in$ {\rm{\textbf{SOC}}}$(-\infty,0)$ and $A, B$ be negative invertible operators. Then
\begin{equation*}
g(A+B)\leq g(A)+g(B),
\end{equation*}
 if and only if $ A(\lambda-B)^{-1}A+B(\lambda-A)^{-1}B\geq 0 $ for $\lambda\in [0,\infty)$.
\end{thm}

\begin{proof}
Assume that $g(A+B)\leq g(A)+g(B)$ for all $g\in\textbf{SOC}(-\infty,0)$. Let $g_{\lambda}(t)=\frac{1}{\lambda-t}$ for  $\lambda\in  [0,\infty)$. Then $g_{\lambda}(t)$ is a positive operator monotone function on $(-\infty,0)$ and so $g_{\lambda}(t)\in \textbf{SOC}(-\infty,0)$. Therefore, we have
\begin{equation*}
g_{\lambda}(A+B)\leq g_{\lambda}(A)+g_{\lambda}(B),
\end{equation*}
which gives
\begin{equation}\label{eq1}
(\lambda -A-B)^{-1}\leq  (\lambda -A)^{-1}+(\lambda -B)^{-1},
\end{equation}
and consequently
\begin{equation*}
(\lambda -A-B)\leq  (\lambda -A-B) (\lambda -A)^{-1} (\lambda -A-B)+(\lambda -A-B) (\lambda -B)^{-1} (\lambda -A-B).
\end{equation*}
We find that 
\begin{align*}
A(\lambda -A)^{-1}=(\lambda -A)^{-1} A=-I+\lambda (\lambda -A)^{-1},\\
B(\lambda -B)^{-1}=(\lambda -B)^{-1} B=-I+\lambda (\lambda -B)^{-1}.
\end{align*}
Compute
\begin{align*}
&(\lambda -A-B) (\lambda -A)^{-1} (\lambda -A-B)\\
&= \lambda^{2}(\lambda -A)^{-1}+\lambda-\lambda^{2}(\lambda -A)^{-1}-\lambda(\lambda -A)^{-1}B+\lambda-\lambda^{2}(\lambda -A)^{-1}-A-\lambda\\
&\,\, +\lambda^{2}(\lambda -A)^{-1}-B+\lambda(\lambda -A)^{-1}B -\lambda B(\lambda -A)^{-1}-B+\lambda B(\lambda -A)^{-1}+B(\lambda -A)^{-1}B,
\end{align*}
and 
\begin{align*}
&(\lambda -A-B) (\lambda -B)^{-1} (\lambda -A-B)\\
&= \lambda^{2}(\lambda -B)^{-1}-\lambda(\lambda -B)^{-1}A+\lambda-\lambda^{2}(\lambda -B)^{-1}-\lambda A(\lambda -B)^{-1}+A(\lambda -B)^{-1}A-A     \\
&\quad +\lambda A(\lambda -B)^{-1}+\lambda-\lambda^{2}(\lambda -B)^{-1}-A+\lambda(\lambda -B)^{-1}A-B-\lambda+\lambda^{2}(\lambda -B)^{-1}.
\end{align*}
Thus, we obtain
\begin{equation*}
(\lambda -A-B)\leq \left\lbrace  \lambda -A-2B+B(\lambda -A)^{-1}B\right\rbrace  +\left\lbrace \lambda -2A-B+A(\lambda -B)^{-1}A\right\rbrace,
\end{equation*}
which gives
\begin{equation}\label{eq2}
A(\lambda -B)^{-1}A+B(\lambda -A)^{-1}B+\lambda -2(A+B)\geq 0,
\end{equation}
for every $\lambda\in  [0,\infty)$ and $A,B<0$. Therefore, we have
\begin{equation*}\label{eq3}
A(\lambda -B)^{-1}A+B(\lambda -A)^{-1}B\geq 0.
\end{equation*}
Conversely, suppose that $A(\lambda -B)^{-1}A+B(\lambda -A)^{-1}B\geq 0$ for $\lambda\in  [0,\infty)$ and $A,B<0$. Thus, we have inequality \eqref{eq2} which is equivalent to \eqref{eq1} . Hence, we get
\begin{equation}\label{eq4}
g_{\lambda}(A+B)\leq g_{\lambda}(A)+g_{\lambda}(B),
\end{equation}
where $g_{\lambda}(t)=\frac{1}{\lambda-t}$. If $g\in\textbf{SOC}(-\infty,0)$, then by \eqref{eq-soc1} we have
\begin{equation*}
g(t)=\alpha + \int\limits_{[0,\infty)} g_{\lambda}(t) d \mu (\lambda),
\end{equation*}
where $\alpha =g(\infty)$. Notice that $\alpha> 0$. Thus, without loss of generality, we can assume
\begin{equation*}
g(t)= \int\limits_{[0,\infty)} g_{\lambda}(t) d \mu (\lambda).
\end{equation*}
It follows from inequality \eqref{eq4} that the desired inequality $g(A+B)\leq g(A)+g(B)$ holds.
\end{proof}

If $\lambda \to 0^{+}$ in \eqref{eq2}, then we have
\begin{cor}
Let $0\neq g\in$ {\rm{\textbf{SOC}}}$(-\infty,0)$ and $A, B$ be negative invertible operators. If $g(A+B)\leq g(A)+g(B)$, then $ AB^{-1}A+B A^{-1}B\leq 0$.
\end{cor}

\begin{cor}
Let $f$ be a positive operator monotone function on $(-\infty,0)$ and $A, B$ be negative invertible operators. Then
\begin{equation*}
 f(A+B)\leq f(A)+f(B),
\end{equation*}
if and only if $A(\lambda -B)^{-1}A+B(\lambda -A)^{-1}B\geq 0$ for $\lambda\in  [0,\infty)$.
\end{cor}

Let $A, B$ be positive invertible operators and $-1\leq p <0$. According to the operator convexity of the function $t^{p}$ on $(0, \infty)$, it is known that
\begin{equation}\label{eq10}
(A+B)^{p}\leq 2^{p-1} (A^{p}+B^{p}).
\end{equation}

\begin{thm}\label{16}
Let  $g$ be a positive operator convex function on $(0, \infty)$  with $g(0^{+})=\lim\limits_{t\to0^+}g(t)=0$.  Put $f(t)=\frac{1}{g(t)}$ on $(0, \infty)$. Then $AB+BA\geq 0$  if and only if 
\begin{equation*}
f(A+B)\leq \frac{1}{4} (f(A)+f(B)),
\end{equation*}
for all positive invertible operators $A,B$.
\end{thm}

\begin{proof}
Suppose $g$ is a positive operator convex function on $(0, \infty)$ with $g(0^{+})=0$.  Then $g$ with $g(0)=g(0^{+})=0$ is a positive operator convex function on $[ 0, \infty)$.  
Now, Theorem \ref{Uchi} shows that  $AB+BA\geq 0$  if and only if
\begin{equation}\label{eq9}
\frac{1}{g(A+B)}\leq \frac{1}{g(A)+g(B)}\leq \frac{1}{4} \left(\frac{1}{g(A)}+\frac{1}{g(B)}\right),
\end{equation}
for all positive operators $A,B$. Notice that the second inequality $\frac{1}{g(A)+g(B)}\leq \frac{1}{4} \left(\frac{1}{g(A)}+\frac{1}{g(B)}\right)$ is obtained from \eqref{eq10} with $p=-1$. Putting $f(t)=\frac{1}{g(t)}$ on $(0, \infty)$ in \eqref{eq9}, we find that $AB+BA\geq 0$  if and only if 
\begin{equation*}
f(A+B)\leq\frac{1}{4} (f(A)+f(B)),
\end{equation*}
for all positive invertible operators $A,B$.
\end{proof}

In comparison with Theorem \ref{3}, Theorem \ref{16} entails the following result.
\begin{cor}
Let $f$ be a positive operator monotone function on $(0, \infty)$. Then
\begin{equation*}
f(A+B)\leq \dfrac{f(A)+f(B)}{4} \leq f(A)+f(B),
\end{equation*}
if and only if  $AB+BA\geq 0$.
\end{cor}

In what follows, our purpose is to obtain a new generalization of Theorem \ref{10} involving the weighted operator arithmetic mean and weighted operator harmonic mean.

\begin{thm}\label{9}
Let $\alpha\in(0,1)$  and $g\neq 0$ be a continuous function on an interval $J$.
\begin{enumerate}[(i)]
\item If $A, B> 0$ and $(\alpha A+(1-\alpha) B)>0$, then
\begin{equation}\label{eq8}
A\nabla_{\alpha} B - A!_{\alpha} B =\alpha (1-\alpha) (A-B) (\alpha A+(1-\alpha)B)^{-1} (A-B).
\end{equation}
\item $g\in$ {\rm{\textbf{SOC}}}$(J)$ if and only if  $g(t)>0$ and
\begin{equation*}
g(A\nabla_{\alpha} B)\leq g(A) !_{\alpha} g(B),
\end{equation*}
for all self-adjoint operators $A,B$ with spectra in $J$.
\end{enumerate}
\end{thm}

\begin{proof}
\begin{enumerate}[(i)]
\item Let $A, B>0$. Put $X=\alpha A+(1-\alpha)B$. We have
\begin{align*}
A\nabla_{\alpha} B - A!_{\alpha} B&=(1-\alpha)A+\alpha B-((1-\alpha)A^{-1}+\alpha B^{-1})^{-1}\\
&=(1-\alpha)A+\alpha B - AX^{-1} B\\
&=(1-\alpha)A+\alpha B-(X+(1-\alpha)(A-B)) X^{-1} (X-\alpha (A-B))\\
&=(1-\alpha)A+\alpha B-X+\alpha (A-B)-(1-\alpha)(A-B)\\
&\qquad +\alpha (1-\alpha)(A-B) X^{-1} (A-B)\\
&= \alpha (1-\alpha)(A-B) X^{-1} (A-B).
\end{align*}

\item Let $0\neq g\in$ {\rm{\textbf{SOC}}}$(J)$. Then  $g(t)>0$ and so $g(A) !_{\alpha} g(B)$ is well-defined  for all self-adjoint operators $A,B$ with spectra in $J$. We have
\begin{align*}
&g(A)\nabla_{\alpha} g(B) - g(A) !_{\alpha} g(B)  \\
&\, =\alpha (1-\alpha) (g(A)-g(B))(\alpha g(A)+(1-\alpha)g(B))^{-1} (g(A)-g(B)) \quad (\text{by} \,\eqref{eq8}) \\
&\, \leq g(A)\nabla_{\alpha} g(B)-g(A\nabla_{\alpha} B)  \quad (\text{by Theorem \ref{8}}).
\end{align*}
Hence, we obtain the desired inequality $g(A\nabla_{\alpha} B)\leq g(A) !_{\alpha} g(B) $. Conversely, assume $g(t)>0$ on $J$ and $g(A\nabla_{\alpha} B)\leq g(A) !_{\alpha} g(B)$ for all self-adjoint operators $A,B$ with spectra in $J$. Then
 \begin{align*}
& g(A)\nabla_{\alpha} g(B)-g(A\nabla_{\alpha} B)\\
&\, \geq g(A)\nabla_{\alpha} g(B) - g(A) !_{\alpha} g(B)\\
 &\, =\alpha (1-\alpha) (g(A)-g(B))(\alpha g(A)+(1-\alpha)g(B))^{-1} (g(A)-g(B))  \quad (\text{by} \,\eqref{eq8}).
 \end{align*}
 Now,  Theorem \ref{8} implies that $g\in$ {\rm{\textbf{SOC}}}$(J)$.
\end{enumerate}
\end{proof}

The parallel sum for positive invertible operators $A$ and $ B$ can be written as follows
\begin{equation}\label{eq5}
\langle (A^{-1}+B^{-1})^{-1}  z,z\rangle = \inf_{x} \left\lbrace \langle Ax,x \rangle + \langle By,y \rangle \vert z=x+y\right\rbrace,
\end{equation}
for all $x, y, z \in \mathcal{H}$. We would remark that \eqref{eq5} is called as the minimum characterization of the parallel sum \cite[Lemma 5.3]{Mond}. 

\begin{lem}\label{12}
Let $A, B, C> 0$  and $\alpha\in (0,1)$. Then $C\leq A!_{\alpha} B$ if and only if
\begin{equation*}
\begin{bmatrix} \alpha A & 2\alpha (1-\alpha) C \\  2\alpha (1-\alpha) C & (1-\alpha) B \end{bmatrix} \geq \alpha (1-\alpha)\begin{bmatrix} C &C \\ C & C \end{bmatrix}.
\end{equation*}
\end{lem}

\begin{proof}
The proof follows the outlines of \cite[Lemma 2.3]{Uchiyama3}. Let $A, B, C>0$ and $\alpha\in (0,1)$. It follows from the minimum characterization of the parallel sum that
\begin{equation}\label{eq6}
\langle ((1-\alpha) A^{-1}+\alpha B^{-1})^{-1}z,z\rangle = \frac{1}{\alpha (1-\alpha)}\inf_{x} \left\lbrace \alpha \langle Ax,x \rangle + (1-\alpha) \langle By,y \rangle \vert z=x+y\right\rbrace ,
\end{equation}
for all $x, y, z \in \mathcal{H}$. Suppose that
\begin{equation*}
\begin{bmatrix} \alpha A & 2\alpha (1-\alpha) C \\  2\alpha (1-\alpha) C & (1-\alpha) B \end{bmatrix} \geq \alpha (1-\alpha)\begin{bmatrix} C &C \\ C & C \end{bmatrix},
\end{equation*}
which is equivalent to 
\begin{equation}\label{eq13}
M=\begin{bmatrix} \alpha A-\alpha (1-\alpha) C  & \alpha (1-\alpha) C \\  \alpha (1-\alpha) C & (1-\alpha) B-\alpha (1-\alpha) C \end{bmatrix} \geq 0.
\end{equation}
Note that $M\geq 0$ if and only if  $ X^{*} M X \geq 0$, where $X=\begin{bmatrix} x  \\ y \end{bmatrix}$ for all $x, y\in \mathcal{H}$. By easy computation, we find that \eqref{eq13} holds if and only if
\begin{equation*}
\langle  C(x-y),(x-y)\rangle \leq \frac{1}{\alpha (1-\alpha)}\left( \alpha \langle Ax,x \rangle + (1-\alpha) \langle By,y \rangle \right),
\end{equation*}
for all $x, y\in \mathcal{H}$. Therefore, applying \eqref{eq6}, we  conclude that $M\geq 0$ if and only if
\begin{align*}
\langle  C(x+(-y)),(x+(-y))\rangle &\leq \frac{1}{\alpha (1-\alpha)}\inf_{x} \lbrace \alpha \langle Ax,x \rangle + (1-\alpha) \langle By,y \rangle  \vert z=x+(-y)  \rbrace \\
&=\langle  ((1-\alpha) A^{-1}+\alpha B^{-1})^{-1} (x+(-y)),(x+(-y))\rangle ,
\end{align*}
which gives $C\leq ((1-\alpha) A^{-1}+\alpha B^{-1})^{-1}$, and consequently  $C\leq A!_{\alpha} B$.
\end{proof}

Now, Theorem \ref{9} and Lemma \ref{12} entail that
\begin{cor}\label{13}
Let  $g\neq 0$ be  continuous function on $J$ and $\alpha\in (0,1)$. Then $g\in$ {\rm{\textbf{SOC}}}$(J)$ if and only if $g(t)>0$ and 
\begin{equation*}
\begin{bmatrix} \alpha g(A) & 2\alpha (1-\alpha) g(A\nabla_{\alpha} B) \\ 2\alpha (1-\alpha) g(A\nabla_{\alpha} B) & (1-\alpha) g(B) \end{bmatrix} \geq \alpha (1-\alpha) \begin{bmatrix} g(A\nabla_{\alpha} B) & g(A\nabla_{\alpha} B) \\ g(A\nabla_{\alpha} B) & g(A\nabla_{\alpha} B) \end{bmatrix},
\end{equation*}
for all self-adjoint operators $A, B$ with spectra in $J$.
\end{cor}

\begin{rem}
If $\alpha=\frac{1}{2}$, Theorem \ref{9} gives \cite[Lemma 2.1 (i)]{Uchiyama3}. We also remark that Theorem \ref{9} and Corollary \ref{13} with $\alpha=\frac{1}{2}$ imply Theorem \ref{10} .
\end{rem}

Finally in this section, we consider relations between strongly operator convex functions and Kwong functions on $(0, \infty)$.

\begin{prop}\label{14}
Let $g\neq 0$ be a continuous function on $(0, \infty)$.
\begin{enumerate}[(i)]
\item If $g\in$ {\rm{\textbf{SOC}}}$(0, \infty)$, then $g(t^{p})\in$ {\rm{\textbf{Kwong}}}$(0, \infty)$, where $-1\leq p\leq 1$.
\item If $g\in$ {\rm{\textbf{Kwong}}}$(0, \infty)$, then $\frac{g(t^{p})}{t^{p}}\in$ {\rm{\textbf{SOC}}}$(0, \infty)$, where $0\leq p\leq\frac{1}{2}$.
\end{enumerate}
\end{prop}

\begin{proof}
\begin{enumerate}[(i)]
\item Let  $ g\in$ {\rm{\textbf{SOC}}}$(0, \infty)$ and $-1\leq p\leq 1$. Then $g(t^{p})$ is a positive operator decreasing or increasing function on $(0, \infty)$. Thus $g(t^{p})$ is a Kwong function on $(0, \infty)$.
\item Let $g$ be a Kwong  function on $(0,\infty)$ and $0\leq p\leq \frac{1}{2}$. It follows from Theorem \ref{2} that $\frac{ g(t^{1/2})}{t^{1/2}}$ is  a positive operator  decreasing function on $(0,\infty)$. Hence  $\frac{g(t^{p})}{t^{p}}=\frac{ g(t^{2p/2})}{t^{2p/2}}$ is also a positive  operator  decreasing. Therefore $\frac{g(t^{p})}{t^{p}}\in\textbf{SOC}(0, \infty)$.
\end{enumerate}
\end{proof}

\begin{ex}
The function $g(t)=\sinh^{-1}(t)$ is  Kwong  on $(0,\infty)$ \cite[Example 2.10]{15}. Let  $A, B>0$ and $0\leq p\leq\frac{1}{2}$. Now,  Proposition \ref{25} and (ii) of Proposition \ref{14} ensure that
\begin{equation*}
\frac{ \sinh^{-1} (A+B)^{p}}{(A+B)^{p} }\leq \frac{\sinh^{-1} (A^p)}{A^{p}}+ \frac{\sinh^{-1} (B^p)}{B^{p}}.
\end{equation*}
\end{ex}

\begin{cor}\label{7}
Let $0\neq g\in$ \rm\textbf{SOC}$(0, \infty)$, $-1\leq p\leq 1$ and  $A$ and $B$ be positive invertible operators. If $AB+BA\geq 0$, then
\begin{equation*}
g(A^{p}) B+B g(A^{p})\geq 0.
\end{equation*}
\end{cor}

\begin{proof}
According to  (i) of Proposition \ref{14}, $g(t^{p})$ is Kwong on $(0, \infty)$. Hence, we deduce the desired inequality by applying \cite[Theorem 2.1]{15} with the Kwong function $g(t^{p})$.
\end{proof}

Uchiyama \cite{Uchiyama2} showed that if $AB+BA\geq 0$ and $0< p, q\leq 1$, then  
$A^{p}B^{q}+B^{q}A^{p}\geq 0$  for positive operators $A,B$.  Corollary \ref{7} gives a new improvement as follows.
\begin{cor}
Let $A$ and $B$ be positive invertible operators and $-1\leq p, q\leq 1$. If $AB+BA\geq 0$, then $A^{-p}B^{-q}+B^{-q}A^{-p}\geq 0$.
\end{cor}

\section{Estimation of lower bounds for strongly operator convex functions}\label{sec3}
We start this section with the following useful lemma.

\begin{lem}\emph{\cite[Lemma 2.1]{MN2012} and \cite[Lemma 3.1]{NM}}\label{15}
Let $A$ and $B$ be positive invertible operators such that $A-B\geq m>0$ for some scalar $m>0$. Then
\begin{align*}
B^{-1}-A^{-1}&\geq m (\Vert A\Vert -m)^{-1} \Vert A \Vert^{-1}\\
B^{-1}-A^{-1}&\geq m \Vert B \Vert^{-1} (\Vert B\Vert +m)^{-1},
\end{align*}
where $\Vert \cdot \Vert  $ is the operator norm.
\end{lem}

If $A-B\geq m>0$, then $\Vert A \Vert\geq \Vert B\Vert+m$. Thus,  Lemma \ref{15} implies 

\begin{cor}\label{19}
Let $A$ and $B$ be positive invertible operators such that $A-B\geq m>0$ for some scalar $m>0$. Then
\begin{equation*}
B^{-1}-A^{-1}\geq m \Vert B \Vert^{-1} (\Vert B\Vert +m)^{-1} \geq  m (\Vert A\Vert -m)^{-1} \Vert A \Vert^{-1}>0.
\end{equation*}
\end{cor}

Following we estimate lower bounds of $g(B)-g(A)$ when $g$ is a strongly operator convex function and $A-B>0$.

\begin{thm} \label{20}
Let $0\neq g\in$ {\rm{\textbf{SOC}}}$(a, \infty)$ with $a>-\infty$ and $A, B$ be self-adjoint operators with spectra in $(a, \infty)$. Then 
\begin{enumerate}[(i)]
\item If $A>B$, then $g(A)<g(B)$.
\item If $A-B\geq m>0$ for some scalar $m>0$, then
\begin{equation*}
g(B)-g(A)\geq g(\sigma_{B})-g(\sigma_{B}+m)\geq g(\sigma_{A}-m)-g(\sigma_{A})>0,
\end{equation*}
in which $\sigma_{A}=\max {\rm sp}(A)$ and $\sigma_{B}=\max {\rm sp}(B)$.
\end{enumerate}
\end{thm}

\begin{proof}
\begin{enumerate}[(i)]
\item Let $A$ and $B$ be self-adjoint operators with spectra in $(a, \infty)$ and $a>-\infty$. If $A>B$, then there is a scalar $m>0$ such that $A-B\geq m>0$. Put $g_{\lambda}(t)=\frac{1}{t-\lambda}$ on $(a, \infty)$ for $\lambda\in (-\infty,a]$. Notice that $(A-\lambda)$ and $(B-\lambda)$ are positive invertible operators and $(A-\lambda)-(B-\lambda)=(A-B)\geq m>0$. By Corollary \ref{19}, we have
\begin{align}\label{eq11}
g_{\lambda}(B)-g_{\lambda}(A)&=(B-\lambda)^{-1}-(A-\lambda)^{-1}\notag\\
&\geq m \Vert B-\lambda \Vert^{-1} (\Vert B-\lambda\Vert +m)^{-1}\\
&\geq m (\Vert A-\lambda\Vert -m)^{-1} \Vert A-\lambda \Vert^{-1}\notag\\
&>0 \notag.
\end{align}
Now, suppose $0\neq g\in$ {\rm{\textbf{SOC}}}$(a, \infty)$. It follows from the integral representation \eqref{eq-soc} that
\begin{equation}\label{eq14}
g(B)-g(A)=\int\limits_{(-\infty ,a]} (g_{\lambda}(B)-g_{\lambda}(A)) d \mu (\lambda)>0,
\end{equation}
and so $g(A)<g(B)$. 
\item  Let $A$ and $B$ be self-adjoint operators with spectra in $(a, \infty)$ such that $A-B\geq m>0$ and $\sigma_{B}=\max {\rm sp}(B)$. To acheive the first inequality, note that
\begin{align*}
\Vert B-\lambda \Vert &=\max \left\lbrace \sigma \vert \,\, \sigma \in {\rm sp}(B-\lambda)\right\rbrace\\
&=\max  \left\lbrace  \sigma-\lambda \vert \,\, \sigma \in {\rm sp}(B)\right\rbrace\\
&=\sigma_{B}-\lambda.
\end{align*}
According to \eqref{eq14} and the first inequality of \eqref{eq11}, we have
\begin{align*}
g(B)-g(A)&\geq m \int\limits_{(-\infty ,a]}  \Vert B-\lambda\Vert^{-1} ( \Vert B-\lambda\Vert +m)^{-1} d \mu (\lambda)\\
&=m \int\limits_{(-\infty ,a]}  (\sigma_{B}-\lambda)^{-1} (\sigma_{B}-\lambda+m)^{-1} d \mu (\lambda)\\
&= \int\limits_{(-\infty ,a]} \left[(\sigma_{B}-\lambda+m)-(\sigma_{B}-\lambda)\right] (\sigma_{B}-\lambda)^{-1} (\sigma_{B}-\lambda+m)^{-1} d \mu (\lambda)\\
&= \int\limits_{(-\infty ,a]} \left[ (\sigma_{B}-\lambda)^{-1}-(\sigma_{B}-\lambda+m)^{-1}\right]  d\mu (\lambda)\\
&=g(\sigma_{B})-g(\sigma_{B}+m) \qquad (\text{by}\,\, \eqref{eq-soc}),
\end{align*}
and so we obtain the first required inequality
\begin{equation*}
g(B)-g(A)\geq g(\sigma_{B})-g(\sigma_{B}+m).
\end{equation*}
To prove the second desired inequality,  let $\sigma_{A}=\max {\rm sp}(A)$. Applying the above computations with the second inequality of \eqref{eq11}, we obtain that
\begin{align*}
g(\sigma_{B})-g(\sigma_{B}+m)&= m \int\limits_{(-\infty ,a]}  \Vert B-\lambda\Vert^{-1} ( \Vert B-\lambda\Vert +m)^{-1} d \mu (\lambda)\\
&\geq m \int\limits_{(-\infty ,a]} (\Vert A-\lambda\Vert -m)^{-1} \Vert A-\lambda \Vert^{-1} d \mu (\lambda)\\
&= m \int\limits_{(-\infty ,a]} \left[  (\sigma_{A}-\lambda-m)^{-1} (\sigma_{A}-\lambda)^{-1}  \right] d \mu (\lambda)\\
&= \int\limits_{(-\infty ,a]} \left[  (\sigma_{A}-\lambda-m)^{-1} -(\sigma_{A}-\lambda)^{-1}  \right] d \mu (\lambda)\\
&=g(\sigma_{A}-m)-g(\sigma_{A}) \qquad (\text{by}\,\, \eqref{eq-soc}).
\end{align*}
Hence, we get
\begin{equation*}
g(\sigma_{B})-g(\sigma_{B}+m)\geq g(\sigma_{A}-m)-g(\sigma_{A}). 
\end{equation*}
The last desired inequality is evidently holds. We have $(\sigma_{A}-m)< \sigma_{A}$ and according to (i), any $ g\in$ {\rm{\textbf{SOC}}}$(a, \infty)$ is strictly operator decreasing on $(a, \infty)$. Therefore  $g(\sigma_{A}-m)-g(\sigma_{A})>0$.
\end{enumerate}
\end{proof}

Theorem \ref{20} is an extension of Corollary \ref{18}. The function $g(t)=\frac{1}{t^p}\in$ {\rm{\textbf{SOC}}}$(0, \infty)$  for $0< p\leq 1$.  Hence, Theorem \ref{20} with $g(t)=\frac{1}{t^p}$ gives the following improvement of Corollary \ref{19}. 
\begin{cor}
Let $A$ and $B$ be positive invertible operators such that $A-B\geq m>0$ for some scalar $m>0$ and $0< p\leq 1$. Then
\begin{equation*}
B^{-p}-A^{-p}\geq  ( \sigma_{B} )^{-p}- (\sigma_{B} +m)^{-p} \geq (\sigma_{A} -m)^{-p}-( \sigma_{A})^{-p}>0,
\end{equation*}
which is equivalent to
\begin{equation*}
B^{-p}-A^{-p}\geq \Vert B \Vert^{-p}- (\Vert B\Vert +m)^{-p} \geq  (\Vert A\Vert -m)^{-p}- \Vert A \Vert^{-p}>0.
\end{equation*}
\end{cor}

As a direct consequence of Theorem \ref{20}, we deduce the following extension of Theorem \ref{17}.  
\begin{cor}\label{23}
Let $f$  be a non-constant and negative operator monotone function on $(a, \infty)$ and $a>-\infty $. If $A$ and $B$ are self-adjoint operators with spectra in $(a, \infty)$ such that $A-B\geq m>0$ for some scalar $m>0$, then
\begin{equation*}
f(A)-f(B)\geq f(\sigma_{B}+m)-f(\sigma_{B}) \geq f(\sigma_{A})-f(\sigma_{A}-m)>0,
\end{equation*}
where $\sigma_{A}=\max {\rm sp}(A)$ and $\sigma_{B}=\max {\rm sp}(B)$.
\end{cor}

\begin{proof}
Put $g(t)=-f(t)$. Then $g(t)>0$ and $g$ is an operator decreasing function on $(a, \infty)$. Thus $(-f(t))\in$ {\rm{\textbf{SOC}}}$(a, \infty)$ and consequently, Theorem \ref{20} gives the desired inequalities.
\end{proof}

\begin{cor}\label{22}
Let $0\neq g\in$ {\rm{\textbf{SOC}}}$(a, \infty)$ with $a>-\infty$ and $A, B$ be self-adjoint operators with spectra in $(a, \infty)$. If $A>B$, then
\begin{equation*}
g(B)-g(A)\geq g(\sigma_{B})-g\left(\sigma_{B}+\Vert (A-B)^{-1}\Vert^{-1}\right) \geq g\left(\sigma_{A}-\Vert (A-B)^{-1}\Vert^{-1}\right)-g(\sigma_{A})>0,
\end{equation*}
where $\sigma_{A}=\max {\rm sp}(A)$ and $\sigma_{B}=\max {\rm sp}(B)$.
\end{cor}

\begin{proof}
If $A>B$, then there is $ m>0$ such that
\begin{equation*}
A-B\geq \frac{1}{\Vert (A-B)^{-1}\Vert}\geq m>0.
\end{equation*}
Putting $m=  \Vert (A-B)^{-1}\Vert^{-1}$ in Theorem \ref{20}, we arrive at the desired inequalities.
\end{proof}

Corollary \ref{22} with the celebrated L\"owner-Heinz inequality leads to the following result.
\begin{cor}
Let $0\neq g\in$ {\rm{\textbf{SOC}}}$(a, \infty)$ with $a>-\infty$ and $A, B$ be self-adjoint operators with spectra in $(a, \infty)$. If $A>B$ and $0<p\leq 1$, then
\begin{equation*}
g(B^{p})-g(A^{p})\geq g(\sigma_{B})-g\left(\sigma_{B}+m\right) \geq g\left(\sigma_{A}-m\right)-g(\sigma_{A})>0.
\end{equation*}
in which $\sigma_{A}=\max {\rm sp}(A^{p})$, $\sigma_{B}=\max {\rm sp}(B^{p})$ and $m=\Vert (A^{p}-B^{p})^{-1}\Vert^{-1}$.
\end{cor}

In the remaining of this section, we study the strongly operator convex functions defined on a left half line.
\begin{thm}\label{21}
Let $0\neq g\in$ {\rm{\textbf{SOC}}}$(-\infty, b)$ with $b<\infty$ and $A, B$ be self-adjoint operators with spectra in $(-\infty, b)$. Then 
\begin{enumerate}[(i)]
\item If $A>B$, then $g(A)>g(B)$.
\item If $A-B\geq m>0$ for some scalar $m>0$, then
\begin{equation*}
g(A)-g(B)\geq g(\tau_{A})-g(\tau_{A}-m)\geq g(\tau_{B}+m)-g(\tau_{B})>0,
\end{equation*}
where $\tau_{A}=\min {\rm sp}(A)$ and $\tau_{B}=\min {\rm sp}(B)$.
\end{enumerate}
\end{thm}

\begin{proof}
\begin{enumerate}[(i)]
\item The proof is similar to that of Theorem \ref{20}. Suppose $A$ and $B$ be self-adjoint operators with spectra in $(-\infty, b)$ and $b<\infty$. Assume that $g_{\lambda}(t)=\frac{1}{\lambda-t}$ on $(-\infty, b)$ for $\lambda\in [b,\infty)$. If $A>B$, then there is a scalar $m>0$ such that $A-B\geq m>0$. Clearly, $(\lambda-B)-(\lambda-A)=(A-B)\geq m>0$ and $(\lambda-A)$ and $(\lambda-B)$ are positive invertible operators. Therefore, Corollary \ref{19} gives
\begin{align}\label{eq12}
g_{\lambda}(A)-g_{\lambda}(B)&=(\lambda-A)^{-1}-(\lambda-B)^{-1}\notag\\
&\geq m \Vert \lambda-A \Vert^{-1} (\Vert \lambda-A\Vert +m)^{-1}\\
&\geq m (\Vert \lambda -B\Vert -m)^{-1} \Vert \lambda -B \Vert^{-1}\notag\\
&>0 \notag.
\end{align}
If $0\neq g\in$ {\rm{\textbf{SOC}}}$(-\infty, b)$, then the integral representation \eqref{eq-soc1} implies that
\begin{equation}\label{eq15}
g(A)-g(B)=\int\limits_{[b,\infty)} (g_{\lambda}(A)-g_{\lambda}(B)) d \mu (\lambda)>0.
\end{equation}
Thus, any non-constant strongly operator convex function on $(-\infty, b)$ is strictly operator monotone on $(-\infty, b)$.
\item Let $A,B$ be self-adjoint operators with spectra in $(-\infty, b)$ such that $A-B\geq m>0$. Put $\tau_{A}=\min {\rm sp}(A)$ and $\tau_{B}=\min {\rm sp}(B)$. We have
\begin{align*}
\Vert \lambda -A \Vert &=\max \left\lbrace \sigma \vert \,\, \sigma \in {\rm sp}( \lambda -A)\right\rbrace\\
&=\max  \left\lbrace  \lambda-\sigma \vert \,\, \sigma \in {\rm sp}(A)\right\rbrace\\
&=\lambda-\tau_{A}.
\end{align*}
Now, applying \eqref{eq15} and \eqref{eq12} via a similar way as in the proof of Theorem \ref{20}, we conclude the desired inequalities. We omit the details of the proof. Indeed, we obtain
\begin{equation*}
g(A)-g(B)\geq g(\tau_{A})-g(\tau_{A}-m)\geq g(\tau_{B}+m)-g(\tau_{B}).
\end{equation*} 
We would remark that $(\tau_{B}+m)>\tau_{B} $ and $ g\in$ {\rm{\textbf{SOC}}}$(-\infty, b)$ is a strictly operator monotone function. Therefore, $g(\tau_{B}+m)-g(\tau_{B})>0$.
\end{enumerate}
\end{proof}

\begin{cor}
Let $0\neq g\in$ {\rm{\textbf{SOC}}}$(-\infty, b)$ with $b<\infty$ and $A, B$ be self-adjoint operators with spectra in $(-\infty, b)$. If $A>B$. Then
\begin{equation*}
g(A)-g(B)\geq g(\tau_{A})-g\left(\tau_{A}-\Vert (A-B)^{-1}\Vert^{-1}\right)\geq g\left(\tau_{B}+\Vert (A-B)^{-1}\Vert^{-1}\right)-g(\tau_{B})>0,
\end{equation*}
in which $\tau_{A}=\min {\rm sp}(A)$ and $\tau_{B}=\min {\rm sp}(B)$.
\end{cor}

\begin{cor}
 Let $0\neq g\in$ {\rm{\textbf{SOC}}}$(-\infty, b)$ with $b<\infty$ and $A, B$ be self-adjoint operators with spectra in $(-\infty, b)$. If $A>B$ and $0<p\leq 1$, then
\begin{equation*}
g(A^{p})-g(B^{p})\geq g(\tau_{A})-g\left(\tau_{A}-m\right)\geq g\left(\tau_{B}+m\right)-g(\tau_{B})>0,
\end{equation*}
where $\tau_{A}=\min {\rm sp}(A^{p})$, $\tau_{B}=\min {\rm sp}(B^{p})$ and $m=\Vert (A^{p}-B^{p})^{-1}\Vert^{-1}$.
\end{cor}

As a direct consequence of Theorem \ref{21}, we estimate lower bounds of $f(A)-f(B)$ when $f$ is an operator monotone function defined on a left half line and $A-B>0$.
\begin{cor}\label{24}
Let $f$  be a non-constant  and positive operator monotone function on $(-\infty, b)$  and $b<\infty $. If $A$ and $B$ are self-adjoint operators with spectra in $(-\infty, b)$ such that $A-B\geq m>0$ for some scalar $m>0$, then
\begin{equation*}
f(A)-f(B)\geq f(\tau_{A})-f(\tau_{A}-m)\geq f(\tau_{B}+m)-f(\tau_{B})>0,
\end{equation*}
where $\tau_{A}=\min {\rm sp}(A)$ and $\tau_{B}=\min {\rm sp}(B)$.
\end{cor}

\section*{Acknowledgments}
N. Gharakhanlu was supported by a grant from the Iran national Elites Foundation (INEF) for a postdoctoral fellowship under the supervision of M. S. Moslehian.

\bibliographystyle{amsplain}

\end{document}